\documentclass[letterpaper,11pt,oneside]{article}
\usepackage[left=2.5cm,right=2.5cm]{geometry}
\usepackage[T1]{fontenc}
\usepackage[utf8]{inputenc}
\usepackage{authblk}
\usepackage{amsmath}
\usepackage{amsthm}
\usepackage[UKenglish]{babel}
\newtheorem{theorem}{Theorem}

\newtheorem{definition}{Definition}
\usepackage{braket,amsfonts}
\usepackage{bm} 
\usepackage{accents}
\newcommand\munderbar[1]{%
  \underaccent{\bar}{#1}}

\usepackage{array} 

\usepackage{amssymb}
\usepackage[caption=false]{subfig}
\usepackage{pgfplots}

\usepackage{algpseudocode}
\usepackage{algorithm}

\usepackage{graphicx,epstopdf}

\usepackage{amsopn}

\usepackage{xspace}
\usepackage{bold-extra}
\usepackage[most]{tcolorbox}

\colorlet{texcscolor}{blue!50!black}
\colorlet{texemcolor}{red!70!black}
\colorlet{texpreamble}{red!70!black}
\colorlet{codebackground}{black!25!white!25}

\providecommand{\keywords}[1]{\textbf{\textbf{Keywords: }} #1}
\newcommand{\dint}{\textnormal{d}}

\patchcmd\newpage{\vfil}{}{}{}
\title{Non-deterministic inference using random set models: \\  theory, approximation, and sampling method} 
\author{Truong-Vinh Hoang} 
\author{Hermann G. Matthies}

\affil{Institute of Scientific Computing, Technische Universit\"at Braunschweig, M\"uhlenpfordtstrasse 23, 38106 Braunschweig, Germany \\
hoang.tr.vinh@gmail.com, truong-vinh.hoang@tu-bs.de, wire@tu-bs.de}

\begin{document}
\maketitle

\begin{abstract}
  A random set is a generalisation of a random variable, i.e. a set-valued random variable.
  The random set theory allows a unification of other uncertainty descriptions such as interval variable, mass belief function in Dempster-Shafer theory of evidence, possibility theory, and set of probability distributions. 
  The aim of this work is to develop a non-deterministic inference framework, including theory, approximation and sampling method, that deals with the inverse problems in which uncertainty is represented using random sets. 
  The proposed inference method yields the posterior random set based on the intersection of the prior and the measurement induced random sets. 
  That inference method is an extension of Dempster's rule of combination, and a generalisation of Bayesian inference as well.  
  A direct evaluation of the posterior random set might be impractical.
  We approximate the posterior random set by a random discrete set whose domain is the set of samples generated using a proposed probability distribution.
  We use the capacity transform density function of the posterior random set for this proposed distribution.
  This function has a special property: it is the posterior density function yielded by Bayesian inference of the capacity transform density function of the prior random set. The samples of such proposed probability distribution can be directly obtained using the methods developed in the Bayesian inference framework.
  With this approximation method, the evaluation of the posterior random set becomes tractable.
\end{abstract}


\section{Introduction}
The inverse problem deals with the identification of the parameters in a computational model given some measurement data. 
It is typical that these parameters are not directly measured but rather their relating quantities which are observable. 
One can make a certain prediction about the parameters and evaluate the observable quantities. 
The difference between the predicted and the actual values of the observable quantities is a measure to evaluate how good a prediction is. 
There are two common methods to solve the inverse problem: deterministic and non-deterministic inferences. 
A deterministic inference targets the prediction that minimizes the different between the predicted and the measured values of the observable quantities. In a non-deterministic inference, the uncertainty -- state of limited knowledge -- about the parameters is updated based on the measurement data. 
For example, in Bayesian inference, the uncertainty is modelled and updated using random variables. 
A requirement to apply Bayesian inference is to formulate the prior uncertainty and the measurement error uncertainty with probability distributions \cite{jaynes2003, stuart2010inverse, Tarantola2005}.
However, in some situations, it could be difficult to derive a probability distribution that can express all the facets of a state of uncertainty. For example, prior knowledge can include nonspecificity, conflict, confusion, vagueness, biases, varying reliability levels of sources, and other types; and measurement data can contain noisy errors and also be coarsened \cite{Ferson2003,ferson2015constructing,beer2013imprecise}.

Several methods were developed to model the uncertainty for different situations, for example: random variable ({rv}), set of possible values e.g. an interval set \cite{moens2004interval}, set of probability distributions (probability box) \cite{kolmogoroff1941confidence, ferson2015constructing, baudrit2006practical}, mass belief function in the Dempster-Shafer (DS) theory of evidence \cite{Dempster1967, shafer1976mathematical}, and random set ({rs}) \cite{matheron1975random, molchanov2005theory}.
Each of these methods have their own advantages in the interpretation the uncertainty. 
In this paper, we focus on the {rs} theory. A {rs} is a set-valued {rv}, i.e. a map from the elementary probability space to the subsets of some domain. The first systematic treatment of the random (closed) set is Matheron \cite{matheron1975random}. The theory was then developed much further by Molchanov in \cite{molchanov2005theory} and by Nguyen in \cite{nguyen2006introduction}. The random set theory is a generalisation of the other listed uncertainty descriptions. Indeed, it is obvious that {rv} is a special case of {rs}. In the cases that the map of a {rs} points to a deterministic set, it becomes the set of possible values. In the context of the evidence theory, the mass belief function can be formulated via a map from a {rv} to subsets of some space. The mass belief function is hence a {rs}. Lastly, the probability box can be represented using a {rs} resulted from the union of the inversion of distribution functions belonging to that box. Thanks to that generality, the {rs} is flexible in modelling different descriptions of uncertainty, while the mathematical formulation remains unchanged.

In this paper, we develop a non-deterministic inference framework in which random sets are used to model uncertainty. Since {rs} theory can formulate all the other listed uncertainty descriptions, such inference framework can be applied for these cases as well as their combinations. The proposed inference is described in short as: the posterior {rs} is the intersection between the two input random sets: the prior {rs} and the {rs} induced by the measurement data. We shall show later in the paper that the proposed inference is a generalisation of Bayesian inference, i.e. when the prior {rs} is simplified to be a {rv}, the posterior {rs} is the posterior {rv} yielded by Bayesian inference. While in the cases that the two input random sets are expressed using evidence theory, the proposed inference method is identical to the Dempster's rule of combination \cite{Dempster1967}. Furthermore, if the input random sets are deterministic sets, the proposed inference yields the intersection of these sets as expected. 

Although the proposed inference rule is quite simple, the computation of its posterior {rs} is problematic. Likewise the Bayesian inference, for a complex computational model, a sampling method is applied to characterize the posterior {rs}. A direct method to determine the set-valued samples, i.e. to solve optimization problems in order to identify each set-valued sample, might be unpractical. In this paper, the posterior {rs} is approximated using a random discrete set whose domain is the set of samples of a proposed distribution. Once the samples of that distribution together with their computational model responses are available, the set-valued samples of the random discrete set are easily identified, and no optimization process is required. Given the set-value samples of the random discrete set, the characteristics of the posterior random sets, e.g. its distribution function and its set-valued expectation, can be estimated.

The choice of the proposed distribution is crucial in order to achieve a good approximation while the computation efficiency remains acceptable, e.g. being comparable with the sampling methods in the framework of Bayesian inference. 
In this work, the capacity transform density function of the posterior {rs}, denoted as $\pi_{T}^a$, is presented and used as the proposed probability density function (pdf). Its has a nice property: the pdf $\pi_{T}^a$ is exactly the posterior pdf yielded by Bayesian inference that updates the capacity transform pdf of the prior {rs}. In other words, we do not need to compute the posterior {rs} in advance, and then evaluate its capacity transform pdf $\pi_{T}^a$. Inversely, one can sample the pdf $\pi_{T}^a$ directly using Bayesian inference, and use the obtained samples to approximate the posterior {rs}. 
The method developed in the framework of Bayesian inference, e.g. Markov chain Monte Carlo (MCMC) \cite{tierney1994markov, gilks1995markov}, Kalman filter \cite{evensen2009data}, inversion via conditional expectation \cite{rosic2013parameter, matthies2016parameter}, can be directly applied to sample  the pdf $\pi_{T}^a$. Furthermore, since $\pi_{T}^a$ is a characteristic of the posterior {rs}, the required number of evaluations of the computational model is smaller than when using other non-informative proposed pdfs while producing a same level of approximation. 

The rest of the paper is organized as follows: In Section \ref{Sec:Backgroud}, the background of {rs} theory is summarized. 
The relations of {rs} theory to the evidence theory (together with possibility theory), and to the probability box are also discussed.
In Section \ref{Sec:inversion_random_set}, the inference method dealing with random sets is given. We shall show that the proposed inference method agrees with the Bayesian one when the prior {rs} is simplified to be a {rv}. In Section \ref{Sec:approximation_of_posterior_randomset}, the approximation of the posterior {rs} using a random discrete set is discussed.  
In Section~\ref{Sec:capacity_pdf}, the capacity transform pdf of the posterior {rs} is defined and is chosen as the proposed pdf. The sampling method of the discrete {rs} that approximates the posterior {rs} is then developed. 
In Section~\ref{Sec:selection_expectation}, the method to estimate the set-valued expectation of the posterior {rs} is given. The developed methods are illustrated through a numerical example in Section~\ref{Sec:numerical_example}. The paper is concluded in Section~\ref{Sec:conclusion}.
\section{Background of the {random set} theory} \label{Sec:Backgroud}
In this section, the background of the {rs} theory is summarized. Details can be found in e.g. \cite{molchanov2005theory}. Because the family of sets is rather rich, it is common to consider random \textit{closed} sets which include the case of random singletons.

\subsection{Random sets} 
Let $(\varOmega, \mathfrak{A}, \mathbb{P})$ be a complete probability space, where the set of elementary events is $\varOmega$, $\mathfrak{A}$ is $\sigma$~-algebra of events, and $\mathbb{P}$ is probability measure. 
The set of closed subsets of $\mathbb{X}\subset \mathbb{R}^n$ is denoted by $\mathfrak{X}$. 
A {rs} ${X}$ is defined as a set-valued measurable map given as
 \begin{equation}
 {X} : \varOmega  \rightarrow \mathfrak{X}.
 \end{equation}
For the sake of simplification, we consider in this paper only integrally bounded random sets, i.e. $\mathbb{E}(\sup_{x\in {X}(\omega)} ||x||)$ is bounded, where $\mathbb{E}$ is the expectation operator. 
When the set ${X}(\omega)$ is singleton set, i.e. it has only one element for all $\omega \in \varOmega$, the {rs} ${X}$ is a (vector valued) {rv}. 
The {rs} is hence considered as a generalisation of {rv}. 
Two measures, the {rs} distribution (RSD) $P_{{X}}$ and capacity functional $T_{{X}}$ of a {rs} ${X}$, are  defined as $P_{{X}}, \; T_{{X}}: \mathfrak{F} \rightarrow [0,1]$ where $\mathfrak{F}$ is $\sigma$-algebra of the set $\mathbb{X}$ such that
\begin{equation}
	P_{{X}}(\mathcal{X}) = \mathbb{P}(\{\omega \;|\; {X}(\omega)\subset \mathcal{X}\}) , 
\end{equation}
and
\begin{equation}
 T_{{X}}(\mathcal{X})= \mathbb{P}(\{\omega \;|\; {X}(\omega) \cap \mathcal{X} \neq \emptyset \}), 
\end{equation}
where $\mathcal{X} \in \mathfrak{F}$ is a measurable set. 
One can directly obtain that 
\begin{equation}
0\leq  P_{{X}} (\mathcal{X}) \leq T_{{X}} (\mathcal{X}) \leq 1.
\end{equation}
Functional $T_{{X}}$ is sub-additive, while $P_{{X}}$ is super-additive,  i.e.
\begin{equation}
\begin{split}
 T_{{X}}(\mathcal{X}_1\cup \mathcal{X}_2) &\leq T_{{X}}(\mathcal{X}_1) + T_{{X}}(\mathcal{X}_2), \quad \\
 P_{{X}}(\mathcal{X}_1\cup \mathcal{X}_2) &\geq P_{{X}}(\mathcal{X}_1) + P_{{X}}(\mathcal{X}_2), 
 \end{split}
\end{equation}
where $\mathcal{X}_1, \mathcal{X}_2 \in \mathfrak{F}$ such that $\mathcal{X}_1\cap \mathcal{X}_2 = \emptyset$.
It is remarked that a probability distribution function is additive. When $T_{{X}} $ and $ P_{{X}}$ are identical, they are a probability distribution function.

A {rv} $x(\omega)$ is a selection {rv} of the rs ${X}(\omega)$ is if $x(\omega) \in {X}(\omega)$ almost surely.  The probability distribution $P$ of a selection {rv} $x(\omega)$ satisfies 
\begin{equation}
	T_{{X}} (\mathcal{X}) \geq P (\mathcal{X}) \geq P_{{X}} (\mathcal{X}).
\end{equation}
The set of all selection random variables is denoted as $\mathcal{S}$. 
Since the {rs} ${X}$ is assumed to be integrally bounded, all the selection random variables of the {rs} ${X}$ are first order rv. The selection expectation $\mathbb{E}_S({X})$ of {rs} ${X}$ is the closure of the set of all expectations of integrable selection random variables, i.e.
\begin{equation}
\mathbb{E}_S({X}) = \text{cl} \{\mathbb{E}(x(\omega))\;|\; x(\omega)\in \mathcal{S}  \}
\label{Eq:slection_expectation}
\end{equation}
where $\mathbb{E}(x(\omega))$ is the expectation of the selection {rv} $x(\omega)$, and $\text{cl}$ is the closure operator. 

In the rest of this section, the relations of the {rs} to the theory of evidence (together with possibility theory), and to the probability box are discussed. 
These two methods are usually applied to model the uncertainties in multi-experts systems or data that contain both epistemic and aleatory errors \cite{ferson2015constructing, beer2013imprecise, sentz2002combination, baudrit2006practical}.
\subsection{Evidence theory and possibility theory}
\label{Sec:DS_theory} 
In the evidence theory, a belief mass function $M$ is defined over $\mathfrak{X}$, $M: \mathfrak{X} \rightarrow [0,1]$, such that 
$M(\emptyset) = 0$ and $\sum_{\mathcal{X}\in \mathfrak{X}} M(\mathcal{X}) =1.$\\
Two measures: belief measure $Bel$ and plausibility measure $Pl$ of a measurable set $\mathcal{X} \subset \mathbb{X}$ are respectively defined as
 \begin{equation}
 Bel (\mathcal{X}) := \sum_{\mathcal{X}'\in \mathfrak{X}} M(\mathcal{X}') \mathbf{1}(\mathcal{X}' \subset \mathcal{X}),
 \end{equation}
and
 \begin{equation}
 Pl (\mathcal{X}) := \sum_{\mathcal{X}'\in \mathfrak{X}} M(\mathcal{X}') \mathbf{1}(\mathcal{X}' \cap \mathcal{X} \neq \emptyset),
 \end{equation}
where $\mathbf{1}(\cdot)$ is a logical operator that yields the unit value if the condition expressed inside the brackets is true, and zero otherwise. 
When $M$ is a consonant mass function on a finite space $\mathbb{X}$, $Pl$ is a possibility measure in the possibility theory \cite{ZADEH19999}. 
Inversely, there is a consonant mass function $M$ such that the possibility measure is the plausibility function corresponding to $M$, (Theorem 2.5.4 page 42 in the reference \cite{halpern2017reasoning}).

One important ingredient of the evidence theory is the Dempster's rule of combination. That combination rule is summarized in Appendix \ref{appendix:Dempster_rule}. 
Based on that rule, we develop the inference method discussed in Sec.~\ref{Sec:inversion_random_set}.
\paragraph{Random set representation of the mass belief function.} Let $\mathfrak{X}_M = \{\mathcal{X}_1, \mathcal{X}_2, \cdots\}$ be the set of alls subsets $\mathcal{X}_i \in \mathfrak{X}$ such that $M(\mathcal{X}_i)>0$, and  $\eta$ be a uniform {rv} $\eta: \; \varOmega \rightarrow [0,1]$. 
Let ${X}: \varOmega \rightarrow \mathfrak{X}_M$ be a {rs} defined as $${X}(\omega) = \mathcal{X}_i\quad \text{if}\quad  \sum_{k=1}^{i-1}M(\mathcal{X}_k) < \eta(\omega) \leq \sum_{k=1}^i M(\mathcal{X}_k),$$
The distribution function $P_{{X}}$ and  the capacity functional $T_{{X}}$ of that {rs} are identical to the belief function, and plausibility function respectively, i.e.  
\begin{equation}
P_{{X}}(\mathcal{X}) = Bel(\mathcal{X}), \quad T_{{X}}(\mathcal{X}) = Pl(\mathcal{X}).
\end{equation} 



\subsection{Set of probability distributions}
\label{Sec:set_probabilities}
A way to describe a set of possible probability distributions of a {rv} is to define the upper and lower bounds on the cdf~\cite{ferson2015constructing, smith1995generalized}. Such expression of the set of possible probability distributions is called probability box. Here, we consider only the cases that the components of the vector $x$ are statistical independent.
Let $\bar{F}$ and $\munderbar{F}$ be the upper and the lower bounds of the cdf, such that $\bar{F} (x) \geq \munderbar{F} (x), \forall x \in \mathbb{X}$. 
The cdf $F$ of the considered {rv} follows the constraint
 \begin{equation}
\munderbar{F} (x) \leq  F(x) \leq \bar{F} (x).
\label{Eq:cdf-box}
\end{equation}
\paragraph{Random set interpretation of the probability box} A {rs} ${X}$ can be constructed from cdfs $\bar{F}$ and $\munderbar{F}$ as
\if01 \footnote{
Vinh's note: when they are not independent, Rosenblatt transform can be used to form the {rs}, however I could not prove that this {rs} satisfy the condition in Eq.~\ref{Eq:distribution_box_Prob}. Using Rosenblatt transform the {rs} ${X}(\omega)$ is obtained as
 \begin{equation}
{X}(\omega) = \{\mathcal{R}_{F}^{-1}(\eta(\omega))  \; | \;\munderbar{F} (x) \leq  F(x) \leq \bar{F} (x) \},
\end{equation}
where $\mathcal{R}_{F}$ is the Rosenblatt transformations \cite{rosenblatt1952remarks} of the distribution function $F$, $\xi(\omega)$ is the uniform {rv} in $[0,1]^n$. When the components $x_i$, $i=1,\dots n$, of the vector $x$ are independent, that transformation can be simplified as}\fi
 \begin{equation}
{X}_i(\omega) = [\bar{F}_i^{-1}(\eta_i(\omega)) \;,\; \munderbar{F}_i^{-1} (\eta_i(\omega))],\quad i=1,\dots n,
\end{equation}
here we abuse the notation $\eta$ and redefine it as the uniform {rv} in $[0,1]^n$. We have 
\begin{equation}
	P_{{X}}( (-\infty,x)) = \munderbar{F}(x),\quad T_{{X}} ((-\infty,x)) = \bar{F}.
	\label{Eq:distribution_box_Prob}
\end{equation}
\section {Interference in the context of {random set} theory}\label{Sec:inversion_random_set} 
In this section, we consider the inference problem in which prior uncertainty is represented using a rs $X(\omega)$. In addition, in order to account for noisy errors and coarsening effects of measurement data, their information is also modelled by a rs $X^d(\omega)$. That inference problem is explained in Sec.~\ref{Sec:inference_prop}. 
The proposed method to update the prior {rs} $X(\omega)$ given the rs $X^d(\omega)$ is discussed in Sec~\ref{Sec:inference_rule}. That inference method is based on the Dempster's rule of combination and agrees with Bayesian inference when the prior random set is simplified to be a {rv}. Furthermore, the proposed method is also linked to Bayesian inference via the capacity transform pdf of random sets. This issue is discussed in Sec~\ref{Sec:capacity_pdf}.   
\subsection{Inference problem}\label{Sec:inference_prop}
Inference problem deals with the identification of parameters, denoted as $x$, given measurements of other quantities $z$ such that the relation between $z$ and $x$ can be represented using a computational model $h$, i.e.
\begin{equation}
z = h(x).
\end{equation}
In practice, the measurement noise is inevitable. Assuming that the noise is additive, the actual measured value $\tilde{z}$ is given as
\begin{equation}
\tilde{z} = h(x) + \epsilon,
\end{equation}
where $\epsilon$ is the actual value of the noise happened when performing the measurement. Furthermore, we deal with the problem that the measurement data do not give directly the value of $\tilde{z}$ but a set $\mathcal{Z}$, e.g. an interval set, such that
\begin{equation}
\tilde{z} \in  \mathcal{Z}.
\end{equation}
That description of measurement data can be encountered in practice when the accuracy of measurement devices, e.g. sensing resolution and/or minimum (maximum) detectable values, are not negligible \cite{ferson2015constructing}. Since the actual value $\epsilon$ is uncertain, it is modelled as a {rv} $\epsilon(\omega)$.
The r.s. of ${X}^d$ induced by that measurement setup is given as
\begin{equation}
{X}^d (\omega) = \{x\in \mathbb{X} \;  | \;  h (x) + \epsilon(\omega) \in \mathcal{Z}\}.
\label{Eq:data_random_set}
\end{equation}
It is remarked that the developed method in this work is still applicable when the uncertainty of $\epsilon$ is modelled as random set.

\subsection{Inference rule}\label{Sec:inference_rule}

In this section, we develop an inference method to update the prior rs $X(\omega)$ given rs ${X}^d (\omega)$. From Section~\ref{Sec:Backgroud}, there are two ways to interpret a {rs}: (i) as a set of selection random variables, and (ii) as a set-valued {rv}. 
Under the former interpretation, one possible non-deterministic inference method is to apply Bayesian inference \textit{independently} to each selection {rv}. 
In this work, we propose an inference method using the latter interpretation. It is described as: the posterior {rs} is the intersection of the prior and the measurement induced random sets. This proposed interference method is based on the Dempster's rule of combination \cite{Dempster1967} summarized in Appendix~\ref{appendix:Dempster_rule}. 
Under the first inference method, a {rs} is treated as the set of independent probability distributions, while the proposed method considers the {rs} as a single piece of information.  
The updated result of the former is less informative than the latter (see Theorem 3.6.6 page 94 in the reference \cite{halpern2017reasoning}). Such comparison of the two methods are illustrated on a simple problem reported in the Appendix \ref{appendix:roburst_bayesian_vs_Dempster}. 
The inference method based on Dempster's rule of combination is given in the follow. 
\begin{definition}[Interference of a {rs} using Dempster's rule of combination] The update of the prior {rs} ${X}$ given the measurement induced {rs} ${X}^d$ is a posterior (updated) {rs} ${X}^{a}: (\varOmega, \mathfrak{A}, \mathbb{P}^{a}) \rightarrow \mathfrak{X}$ defined as
\begin{equation}
 {X}^{a}(\omega) := {X} (\omega)\cap {X}^d (\omega),  
 \label{Eq:Dempster's_rule}
\end{equation}
and probability $\mathbb{P}^{a}$ is updated as
\begin{equation}
\mathbb{P}^{a}(\dint \omega) 
 = \dfrac{\mathbb{P}(\dint \omega)\mathbf{1}({X}^a(\omega) \neq \emptyset)} { 1- K}.
 \label{Eq:Dempster's_rule_normalized_Pr}
\end{equation}
 where $K$ is the degree of conflict and given by
\begin{equation}
K = 1- \int_\varOmega \mathbf{1}({X}^a(\omega) \neq \emptyset) \mathbb{P}(\dint \omega).
\end{equation}

\label{Def:Dempster's_rule}
\end{definition}
The function $\mathbf{1}({X}^a(\omega)\neq \emptyset)$ in Eq.~(\ref{Eq:Dempster's_rule_normalized_Pr}) is interpreted as a likelihood function. The update of $\mathbb{P}^{a}$ is required to rule out empty sets, i.e. $P_{X^a}(\emptyset) = 0$, while the normalised property, $P_{X^a}(\mathbb{X}) = 1$, is conserved. The larger the value $K$, the more significant the conflict between the prior knowledge and the measurement data becomes. If $K = 1$, then the prior and the measurement induced random sets are said to be in total conflict and no interference is possible. 
The updated {rs} ${X}^{a}(\omega) $ is simplified to be {rv} in following cases: the prior {rs} is a {rv}; or the set $\mathcal{Z}$ has only one member and the function $h$ is strictly monotonic on ${X}(\omega)$ almost surely.
Furthermore, it is observed that ${X}^{a}(\omega) \subset {X}(\omega)$. 
Hence, in a sequential update \cite{evensen2009data}, i.e. the ${X}^{a}(\omega)$ becomes the prior {rs} when new data are available, the final updated {rs} might also be simplified as a {rv}.





\subsection*{Relation with Bayes's rule.} We show in the following that the proposed interference method of {rs} agrees with Bayesian inference when the prior {rs} is a {rv}. 
In this case, the prior is modelled as a rv, denoted as $x(\omega)$, 
the update method using Dempster's rule yields a posterior {rv} $x^{a}: (\varOmega, \mathfrak{A}, \mathbb{P}^a) \rightarrow \mathbb{X}$ given by
\begin{equation}
 x^{a}(\omega) = 
 \begin{cases}
   x(\omega) \quad  &\text{if} \quad  x(\omega) \in X^d(\omega)\\
   \emptyset \quad &\text{otherwise}
 \end{cases},
 \label{Eq:x^a}
 \end{equation}
where the probability $\mathbb{P}^a$ is given as
\begin{equation}
 \mathbb{P}^a(\dint \omega) 
 = \dfrac{\mathbb{P}(\dint \omega) \mathbf{1}(x^{a}(\omega) \neq \emptyset)} 
 {\int_{\varOmega}  \mathbf{1}(x^{a}(\omega) \neq \emptyset) \mathbb{P}(\dint \omega)}.
 \label{Eq:DS_rule_to_Bayes}
\end{equation}
The following theorem shows that the r.v. $x^a(\omega)$ is the Bayesian update of the prior r.v. $x(\omega)$. 
\begin{theorem}
The pdf $\pi{^a}$ of the {rv} $x^a(\omega)$ defined in Eq.~(\ref{Eq:x^a}) is the posterior pdf yielded using the Bayes's rule as 
\begin{equation}
\pi{^a} (x) = \dfrac{\pi(x) \mathcal{L}(x)}{\int_{\mathbb{X}} \pi(x) \mathcal{L}(x) d x},
\label{Eq:Bayes_rule_from_Dempster_rule}
\end{equation}
where $\pi(x)$ is the pdf of the prior {rv} $x(\omega)$, and $\mathcal{L}(x)$ is the likelihood function
  \begin{equation}
  \mathcal{L}(x) = \int_{\varOmega} \mathbf{1}_{ X^d(\omega)}(x)  \mathbb{P}(\dint \omega),
  \label{Eq:likelihood}
  \end{equation}
  where $\mathbf{1}_{ \mathcal{X}}(x) := \mathbf{1} (x\in \mathcal{X}) $ is the characteristic function.
  \label{lemma:bayes_dempster_rules}
\end{theorem}
It is noted that to avoid Borel–Kolmogorov paradox, i.e. $\pi{^a} (x)$ is undefined, we consider here the case that the random set $X^d(\omega)$ satisfying $X^d(\omega) = \text{cl}\{\;\text{int}(X^d(\omega))\}$, where $\text{int}(X^d(\omega))$ is the set of interior points of $X^d(\omega)$, almost surely. For the special case that the {rs} $X^d(\omega)$ is a {rv}, we mention it explicitly. This assumption is also applied for random sets $X(\omega)$, and $X^a(\omega)$.  

The proof of Theorem~\ref{lemma:bayes_dempster_rules} is given in Appendix~\ref{appendix:Proof_bayes_dempster_rules}. 
Theorem~\ref{lemma:bayes_dempster_rules} shows that the proposed inference agrees with the Bayesian method when the prior uncertainty is modelled using a probability distribution. 
We shall show later in Section~\ref{Sec:capacity_pdf} that the likelihood function $\mathcal{L}(x)$ defined in the Eq.~(\ref{Eq:likelihood}) is proportional to the capacity transform pdf of the measurement induced {rs} ${X}^d$. 

\subsection{Capacity transform density function}
\label{Sec:capacity_pdf}
The capacity transform pdf of {rs} ${X}$ is defined as
\begin{equation}
\pi_{T}(x) 
= \dfrac{\int_{\varOmega} \mathbf{1}_{ X(\omega)}(x) \mathbb{P}(\dint\omega)}
{\int_{\mathbb{X}} \int_{\varOmega} \mathbf{1}_{ X(\omega)}(x)  \mathbb{P}(\dint\omega) \dint x} .
\end{equation}
For $\pi_{T}(x)$ to be well-defined, it is required that $0 < \int_{\mathbb{X}} \int_{\varOmega} \mathbf{1}_{ X(\omega)}(x)  \mathbb{P}(\dint\omega) \dint x < \infty$. 
In the context of DS theory of evidence, this pdf is called plausibility transform pdf \cite{voorbraak1989computationally,cobb2006plausibility}. When ${X}(\omega)$ is a {rv} $x(\omega)$, the capacity transform pdf $\pi_{T}(x)$ of the random set $B(x(\omega),r)$--closed balls centred at $x(\omega)$ and having radius $r$--converges to the pdf of $x(\omega)$ as $r \rightarrow 0$.

In a similar way of deriving $\pi_{T}(x)$, the capacity transform pdf $\pi_{T}^a$ of the updated {rs} ${X}^{a}$ given in Eq.~(\ref{Eq:Dempster's_rule}) is defined as 
\begin{equation}
\pi_{T}^a(x) \propto \int_{\varOmega} \mathbf{1}_{ X^a(\omega)}(x) \; \; \mathbb{P}^a(\dint\omega).
\label{Eq:capacity_pdf_updated_rs}
\end{equation}
The capacity transform pdf $\pi_{T}^d$ of the measurement induced {rs} ${X}^d$, see Eq.~(\ref{Eq:data_random_set}), is derived as
\begin{equation}
\pi_{T}^d(x) \propto \int_{\varOmega} \mathbf{1}_{ X^d(\omega)}(x) \; \; \mathbb{P}(\dint \omega) 
= \int_{\varOmega} \mathbf{1}_{\mathcal{Z}}(h(x)+ \epsilon(\omega))\; \mathbb{P}(\dint\omega)  .
\label{Eq:capacity_pdf_data_rs}
\end{equation}
It is remarked that the right hand side of Eq.~(\ref{Eq:capacity_pdf_data_rs}) is the likelihood function $\mathcal{L}$ defined in the Eq.~(\ref{Eq:likelihood}).
The relation of $\pi_{T}^a$ to $\pi_{T}$ and $\pi_{T}^d$ (or $\mathcal{L}$)  is given in the following theorem. 
\begin{theorem} [Capacity transformed pdf of the posterior {rs}]
The capacity transformed pdf $\pi_{T}^a$ of the posterior {rs} ${X}^{a}$ is the posterior pdf obtained by Bayesian inference with prior pdf $\pi_{T}$ and 
the likelihood function $\mathcal{L}(x)$ given by Eq.~(\ref{Eq:likelihood}), that is
\begin{equation}
\pi_{T}^a (x) = \dfrac{\pi_{T}(x) \mathcal{L}(x)} {\int_\mathbb{X} \pi_{T}(x) \mathcal{L}(x) \dint x}.
\label{Eq:capacity_tranformed_pdf}
\end{equation}
\label{TH:capacity_tranformed_pdf}
\end{theorem}
\begin{proof}
Inserting the expressions of ${X}^a$ in  Eq.~(\ref{Eq:Dempster's_rule}) and $\mathbb{P}^a$ in Eq.~(\ref{Eq:Dempster's_rule_normalized_Pr}) into Eq.~(\ref{Eq:capacity_pdf_updated_rs}) we have
 \begin{equation}
 \pi_{T}^a (x) \propto \int_{\varOmega} \mathbf{1}_{ X(\omega)}(x) \mathbf{1}_{ X^d(\omega)}(x) \mathbb{P}(\dint \omega).
 \end{equation}
As ${X}(\omega)$ and ${X}^d(\omega)$ are independent, we have
 \begin{equation}
\int_{\varOmega} \mathbf{1}_{ X(\omega)}(x) \mathbf{1}_{ X^d(\omega)}(x) \mathbb{P}(\dint \omega) = \int_{\varOmega} \mathbf{1}_{ X(\omega)}(x) \mathbb{P}(\dint \omega) \int_{\varOmega} \mathbf{1}_{ X^d(\omega)}(x) \mathbb{P}(\dint \omega) = \pi_{T}(x) \mathcal{L}(x).
 \end{equation}
The expression of $\pi_{T}^a$ can be rewritten as in the Eq.~(\ref{Eq:capacity_tranformed_pdf}).
\end{proof}

Using Theorem~\ref{TH:capacity_tranformed_pdf}, it is not required to compute ${X}^a$ explicitly in advance to evaluate the updated capacity transform pdf $\pi_{T}^a$. 
Instead, we can sample the pdf $\pi_{T}^a$ directly  using Bayesian reference, and  these samples are then used to approximate the posterior {rs} ${X}^{a}$ in Section~\ref{Sec:approximation_of_posterior_randomset}.



\section{Approximation of the posterior {rs} using a random discrete set}\label{Sec:approximation_of_posterior_randomset}
To reduce the computational burden, we approximate the posterior {rs} ${X}^{a}$ by a random discrete set. 
Instead of searching for all members of the set ${X}^{a}(\omega)$, 
we limit them only to be elements of a discrete set $\{x^{(1)}, \dots, x^{(\kappa)} \} \subset \mathbb{X}$,
which are generated from a proposed pdf $\pi^{e}$ over the domain $\mathbb{X}$, such that $\pi^{e}(x)>0,\; \forall x \in {X}^{a}(\omega)$ almost surely. There are several ways to choose the proposed pdf $\pi^{e}$. For example, one can use a uniform distribution (if $\mathbb{X}$ is bounded), or an unbounded distribution with a large (co)variance (if $\mathbb{X}$ is unbounded), or the distribution of a selection {rv} of prior {rs}. In these examples, the choices of $\pi^{e}$ are non-informative since they do not account for the measurement data. 
Here we use the pdf $\pi_T^a$ as the proposed pdf $\pi^e$. With this choice for $\pi^e$, we have an informative proposed pdf, while the computational methods that are well-developed in the framework of Bayesian inference can be directly applied to obtain pdf $\pi_T^a$ using Theorem \ref{TH:capacity_tranformed_pdf}. 

The approximation using the random discrete set, denoted as $\hat{X}^{a, \kappa}$, is formulated as 
\begin{equation}
{X}^{a} (\omega) \approx  \hat{X}^{a, \kappa} (\omega ) := {X}^{a} (\omega) \cap  \{x^{(1)}, \dots, x^{(\kappa)}\}. 
\label{Eq:intersection_MC_approximate}
\end{equation}
Using the definition of ${X}^{a} (\omega)$ in Eq.~(\ref{Eq:Dempster's_rule}), its approximated set $\hat{X}^{a,\kappa}$ can be expressed as
\begin{equation}
\hat{X}^{a,\kappa} (\omega) =  \{x^{(i)}\in \{x^{(1)}, \dots, x^{(\kappa)}\}  \;|\; x^{(i)} \in {X}(\omega) \cap  {X}^d(\omega) \}.	
\label{Eq:approximated_set}
\end{equation}
The larger the number $\kappa$, the better the approximation. 
\if00
Such approximation is verified by the following Theorem.
\begin{theorem}[Discrete set approximation of bounded set using samples of a probability distribution]
	Let $ \mathcal{X}\neq \emptyset$ be a bounded set that contains no isolated point and $\pi^{e}(x)$ be a pdf such that $\pi^{e}(x)>0,\;\forall x \in \mathcal{X}$, and $\{x^{(1)}, \dots, x^{(\kappa)}\}$ be its $\kappa$ samples, then the Hausdorff distance between the set $\mathcal{X}$ and the set $\hat{\mathcal{X}} := \mathcal{X} \cap \{x^{(1)}, \dots, x^{(\kappa)}\}$ converges to 0 as $\kappa \rightarrow \infty$ almost surely. 
	\label{TH:condtion_of_proposed_pdf}
\end{theorem}
It is remarked that the Hausdorff distance between two sets is zero only if they have an identical closure. The proof of Theorem~\ref{TH:condtion_of_proposed_pdf} is given in the Appendix~\ref{appendix:condtion_of_proposed_pdf}. 
\fi

\subsection{Sampling method for the posterior {random set}} 
The $\kappa$ samples $x^{(1)}, \dots, x^{(\kappa)}$ of the pdf $\pi_{T}^a$ can be generated using classical methods in the Bayesian inference framework. In this work, we use the Metropolis Hasting MCMC \cite{tierney1994markov, gilks1995markov} algorithm for this task. Note that the algorithm provides not only the samples $x^{(1)}, \dots, x^{(\kappa)}$ but also their model responses $h(x^{(1)}), \dots, h(x^{(\kappa)})$. 

\paragraph{MC simulation of the random discrete set $\hat{X}^{a,\kappa}$.}
Given  the samples $x^{(1)}, \dots, x^{(\kappa)}$ and their model responses, a MC simulation is then applied to obtain the $N$ samples, $\hat{X}^{a,\kappa \; (i)}$ where $i = 1, \dots, N$, of the approximated random discrete set $\hat{X}^{a, \kappa}$ expressed in Eq.~(\ref{Eq:approximated_set}). 
This MC simulation is reported in Algorithm~\ref{algorith:approximate_random_set}. 
\begin{algorithm}
	\caption{MC simulation to sample the approximated posterior {rs} $\hat{X}^{a}$}\label{algorith:approximate_random_set}
	\begin{algorithmic}[1]
		
		\State Input:  $\kappa$ samples  $x^{(1)}, \dots, x^{(\kappa)}$ of the pdf $\pi_T^a$ and their model responses $h(x^{(1)}), \dots, h  (x^{(\kappa)})$ using a MCMC simulation.

		\State Generate $N_1$ samples of prior {rs} $\mathcal{X}^{(1)}, \dots, \mathcal{X}^{(N_1)}$
		\State Generate $N_2$ samples of measurement error  $\epsilon^{(1)}, \dots, \epsilon^{(N_2)}$
		\State $N\gets 0$
		\For{$\mathcal{X}^{(i)}$ in $\mathcal{X}^{(1)}, \dots, \mathcal{X}^{(N_1)}$}
		
		\For{$\epsilon^{(j)}$ in $\epsilon^{(2)}, \dots, \epsilon^{(N_2)}$}
		
		\State $\mathcal{X}^* \gets \{x\in \{x^{(1)}, \dots, x^{(\kappa)}\} \cap  \mathcal{X}^{(i)} \;|\quad  h(x) +\epsilon^{(j)} \in \mathcal{Z}\}$
		
		\If {${X}^* \neq \emptyset$}
		\State $N \gets N + 1$
		\State ${X}^{a, \kappa \;(N)} = {X}$
		
		\EndIf
		\EndFor
		\EndFor
		\State \textbf{return} $\hat{X}^{a,\kappa \;(i)}$ where $i = 1, \dots, N$
	\end{algorithmic}
\end{algorithm}

From the samples, $\hat{X}^{a,\kappa \; (i)}$ where $i = 1, \dots, N$, the RSD $P_{{X}^a}$ and the capacity functional $T_{{X}^a}$ of the posterior {rs} ${X}^a$ can be approximated respectively as
\begin{equation}
	P_{{X}^a}(\mathcal{X}) \approx \dfrac{1}{N}\sum_{i=1}^{N} \mathbf{1}(\hat{X}^{a, \kappa \;(i)} \subset \mathcal{X} ), \quad
	T_{{X}^a}(\mathcal{X}) \approx \dfrac{1}{N} \sum_{i=1}^{N} \mathbf{1}(\hat{X}^{a, \kappa \;(i)} \cap \mathcal{X} \neq \emptyset)	
	\label{Eq:estimate_capacity}.
\end{equation}

\textbf{Remarks}: The evaluation of the model $h$ is only required for the MCMC algorithm to sample $x^{(1)}, \dots, x^{(\kappa)}$, but not in the later MC simulation summarized in Algorithm~\ref{algorith:approximate_random_set}. In other words, the latter MC simulation is independent of the complex of the model $h$. Therefore the computational cost of our method is comparable with the classical methods used for the Bayesian inference.  With the proposed method, we do not need an optimization process to find member of ${X}^{a}(\omega)$. 
Furthermore, when ${X}^{a}(\omega)$ becomes a {rv}, its pdf is the capacity transform pdf $\pi_{T}^a$, and therefore $x^{(1)}, \dots, x^{(\kappa)}$ are its samples. 


\section{Set-valued selection expectation of posterior {random set}}
\label{Sec:selection_expectation}
In this section, the support function of a given set is introduced. We use these functions as the mean to evaluate the set-valued selection expectation of the posterior {rs} (the definition of selection expectation is formulated in Eq~(\ref{Eq:slection_expectation})).  \\
A support function $\gamma: \mathfrak{X}\times \mathbb{S}^{n-1} \rightarrow  \mathbb{R} $ is defined as 
\begin{equation}
\gamma(\mathcal{X}, \nu) =  \sup_{x \in \mathcal{X}}\ {\nu \cdot x}.
\end{equation}
where $\nu$ is a vector on the unit sphere $\mathbb{S}^{n-1}$ and $\cdot$ is the scalar product.
Applying that support function to the {rs} ${X}^a$, we obtain the scalar-valued {rv} $\gamma({X}^a(\omega), \nu)$. 
\begin{theorem}
If the basic probability space is non-atomic, the selection expectation $\mathbb{E}_S({X}^a)$ is a convex set, and 
\begin{equation}
\gamma(\mathbb{E}_S({X}^a), \nu) = \mathbb{E} (\gamma({X}^a, \nu)).
\label{Eq:support_function_slection_expectation}
\end{equation}
\label{Prop:convex_selection_expectation}
\end{theorem}
The proof of Theorem \ref{Prop:convex_selection_expectation} can be found in the Chapter 2 of reference \cite{molchanov2005theory} (Theorem 1.26).
Thanks to Theorem \ref{Prop:convex_selection_expectation}, the set-valued selection expectation $\mathbb{E}_S ({X}^a)$ can be obtained via the probabilistic expectation of {rv} $\gamma({X}^a(\omega), \nu)$, e.g. using MC method. From $N$ samples $\hat{X}^{a,\kappa \; (1)}, \dots, \hat{X}^{a, \kappa\;(N)}$ of random discrete set of $\hat{X}^{a,\kappa}$ obtained using the Algorithm~\ref{algorith:approximate_random_set}, the expectation of {rv} $\gamma({X}^a(\omega), \nu)$ can be evaluated as
\begin{equation}
\mathbb{E} (\gamma({X}^a, \nu)) \approx \mathbb{E} (\gamma(\hat{X}^{a, \kappa}, \nu)) = \lim _{N\rightarrow \infty} \dfrac{1}{N}\sum_{i=1}^N \gamma(\hat{X}^{a, \kappa\;(i)}, \nu).
\label{Eq:approximate_expectation_set}
\end{equation}
From the Theorem~\ref{Prop:convex_selection_expectation}, if the elementary probability space is non-atomic, the set-valued expectation of the posterior {rs} can be identified as,
\begin{equation}
\mathbb{E}_S({X}^a) = \cap_{\nu \in \mathbb{S}^{n-1}} \{\;x \in \mathbb{X}: \quad \nu\cdot x \leq \mathbb{E} (\gamma({X}^a, \nu)) \;\}.
\end{equation} 
\section{Numerical example}
\label{Sec:numerical_example}
\subsection{Problem setting}To illustrate the developed method, the truss system, see Fig.~\ref{Fig:bridge_truss}, is considered. 
For the sake of simplification, the inference is performed on two parameters: the stiffness $E$ of the horizontal beams, 
and the applied forces $q$. In terms of notation, these parameter are sorted into the vector $x$, i.e. $x = [E, q]$. We fix the other parameters, i.e. the bar cross area, the stiffness of the diagonal bar, as constants. 
We shall use the \textit{virtual} measurement data of the nodal vertical displacements to perform the inference. 
\begin{figure} [th]
  \centering
  \includegraphics[scale=0.55]{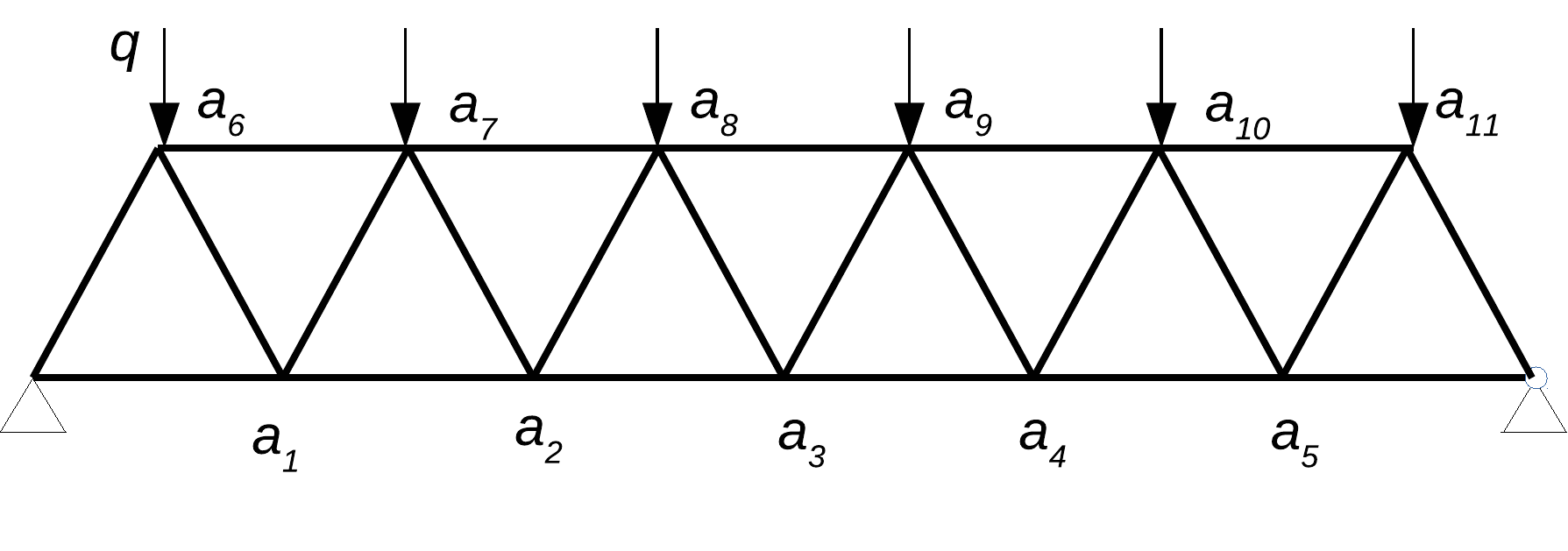}
  \caption{Truss system}
  \label{Fig:bridge_truss}
\end{figure}

The truss system can be solved using a finite element (FE) model of bar elements as
\begin{equation}
u =  [\mathbf{A}(x)]^{-1}f(x)
\label{Eq:FE_model}
\end{equation}
where $u$ is the vector of nodal displacements, $\mathbf{A}$ is the stiffness matrix depending on $E$, and $f$ is the nodal vector of applied forces $q$. That FE model is represented as a function $h$, i.e. $h(x)= [\mathbf{A}(x)]^{-1}f(x)$. 
\paragraph{Prior {rs}}
 The prior {rs} of $x= [E, q]$ is expressed as:
\begin{itemize}
\item  the {rs} of the (dimensionless) stiffness $E$ is expressed using a probability box where the upper and lower bounds of its cdf are the lognormal distributions $L_1(m_{E} = 0.9, \sqrt{ v_{E} } = 0.1)$, $L_2(m_{E} = 1, \sqrt{v_{E}} = 0.11)$ respectively, where $m_{E}$ is the mean, and $v_{E}$ is the variance;
\item the randomness of the (dimensionless) applied force $q$ is expressed using mass belief function as: the possible events are $\mathcal{Q}_1 = [0.77, 0.92]$, $\mathcal{Q}_2 = [0.85, 0.98]$, $\mathcal{Q}_3 = [0.96, 1.08]$, and their masses are given as $M_Q(\mathcal{Q}_1) = 0.3$ and $M_Q(\mathcal{Q}_2) = 0.3$, $M_Q(\mathcal{Q}_3) = 0.4$.
\end{itemize}
We also assume that they are independent. The prior {rs} of $E$ can be encountered in practice as the bounds of prior cdf. While the prior {rs} of $q$ might result when collecting information from different sources in which information are expressed using intervals.

\textbf{Remark.} The prior description of the uncertainties of the parameters is  a mix of probability box and mass belief function. However, under the umbrella of a {rs}, their formulations are similar. That is one advantage when working with {rs}. 
   
\paragraph{Virtual measurement data} We set a vector $x_{t}$ of parameters as the truth and perform measurements \textit{virtually}, i.e. the equation (\ref{Eq:FE_model}) is solved with $x_t$ to obtain the displacement vector $u_t = h(x_t)$. 
The vector $u_t$ is then perturbed by adding the random errors $\epsilon$ which are modelled following Gaussian distributions, $\mathcal{N} (0, 1)$, and are assumed to be independent. The observation sets $\mathcal{Z}$ are derived to model the sensing resolution of measurement devices, which is assumed to be a unit in this example. The virtual measurement data of the displacements at the points a$_1$, \dots, a$_{11}$, see Fig.~\ref{Fig:bridge_truss}, are reported in Tab.~\ref{Tab:observation_data}. We consider two cases: (i) the inference is performed based on one measurement datum at points a$_1$, and (ii) the inference is performed using all the virtual measurement data. 
\begin{table} [th]

  \caption{Description of the \textit{virtual} measurements.}
   \begin{tabular}{ l | c | c | c | c | c }
     position   & a$_1$ & a$_2$ & a$_3$ & a$_4$ & a$_5$\\
        \hline
  true value  $u_t$&  -4.3959 &  -5.9547 &  -5.3349 &  -3.8462 &  -1.9231 \\
  $ \tilde{u} = u_t + \epsilon $ & -5.2414 &  -5.7764  &   -6.1868  & -2.9703 &  -3.8885 \\
  observation sets $\mathcal{Z}$ & [-6, -5] &  [-6, -5] &  [-7, -6]  & [-3, -2]  & [-4, -3]\\
  \end{tabular} 
~\\
~\\
~\\
   \begin{tabular}{ l | c | c | c | c | c |c }
	position   & a$_6$ & a$_7$ & a$_8$ & a$_9$ & a$_{10}$ & a$_{11}$\\
	\hline
	true value  $u_t$&  -2.5000 &-5.7488 &  5.7263 &  -4.6177 &  -2.8575 &  -0.8801 \\ 
	$ \tilde{u} = u_t + \epsilon $ & -0.7187 &-6.7999 &  -7.0940  &  -4.5517  & -2.0691 & 0.9171 \\ 
	observation sets $\mathcal{Z}$ & [-1, 0] & [-7, -6] &  [-8, -7] &  [-5, -4]  & [-3, -2] &  [0, 1]\\ 
\end{tabular}
  \label{Tab:observation_data}
\end{table}

\subsection{Numerical results}
\subsubsection*{Samples of the updated capacity transform pdf $\pi_T^a$}
The marginal capacity transform pdf of the prior {rs} of $E$, $\pi_{T_E}(x_1)$, can be evaluated as
\begin{equation}
\pi_{T_E}(x_1) \propto 
\bar{F}_E (x_1) - \munderbar{F}_E (x_1) 
\end{equation}
where $\bar{F}_E$ and $\munderbar{F}_E$ are its upper and lower cdf bounds of elastic modulus $E$. The marginal capacity transform pdf of the prior {rs} of $q$, $\pi_{T_2}(x_2)$, is evaluated as, 
\begin{equation}
\pi_{T_q}(x_2) \propto \sum_{i=1}^3 M_Q(\mathcal{Q}_i) \mathbf{1}_{\mathcal{Q}_i}(x_2).
\end{equation}
As $E$ and $q$ are independent, $\pi_{T}(x) = \pi_{T_E}(x_1) \pi_{T_q}(x_2)$.
The likelihood function, see Eq.~(\ref{Eq:likelihood}), in this case is simplified  as
\begin{equation}
\mathcal{L}(x) = \prod_{i=1}^{n_d} \bigr(F_{\epsilon} (u_i(x)-\munderbar{z}_i) -  F_{\epsilon} (u_i(x)-\bar{z}_i)\bigr),
\end{equation}
where $n_d$ is number of measurement data used, $F_{\epsilon}$ is the cdf of Gaussian distribution $\mathcal{N}(0,1)$, $u_i$ is the vertical displacement at the point a$_i$ computed using FE model in Eq.~(\ref{Eq:FE_model}), $\munderbar{z}_i = \min(\mathcal{Z}_i) $ and $\bar{z}_i = \max(\mathcal{Z}_i) $, where $\mathcal{Z}_i$ is the observed interval of the displacement at point a$_i$, see Tab. \ref{Tab:observation_data}. 

Using a MCMC simulation, the $\kappa$ samples $x^{(1)},\dots, x^{(\kappa)}$ of the posterior capacity transform pdf $\pi_T^a$ are obtained. From these samples, the pdf $\pi_T^a$ is estimated and illustrated in Fig.~\ref{Fig:capacitypdf}\if01 \footnote{Vinh note: change the axis label $E_h$ to $E$}\fi. It is observed that the updated capacity transform pdf converges to the truth parameter vector $x_t$ when more data are involved as expected. 
Note that at this step, the model responses $u^{(i)} = h (x^{(i)})$ where $i = 1, \dots, \kappa$ are also obtained. 

\begin{figure} [th]
     \centering
     \subfloat[Marginal capacity transfrom pdf of $E$ ]{\includegraphics[scale=0.45]{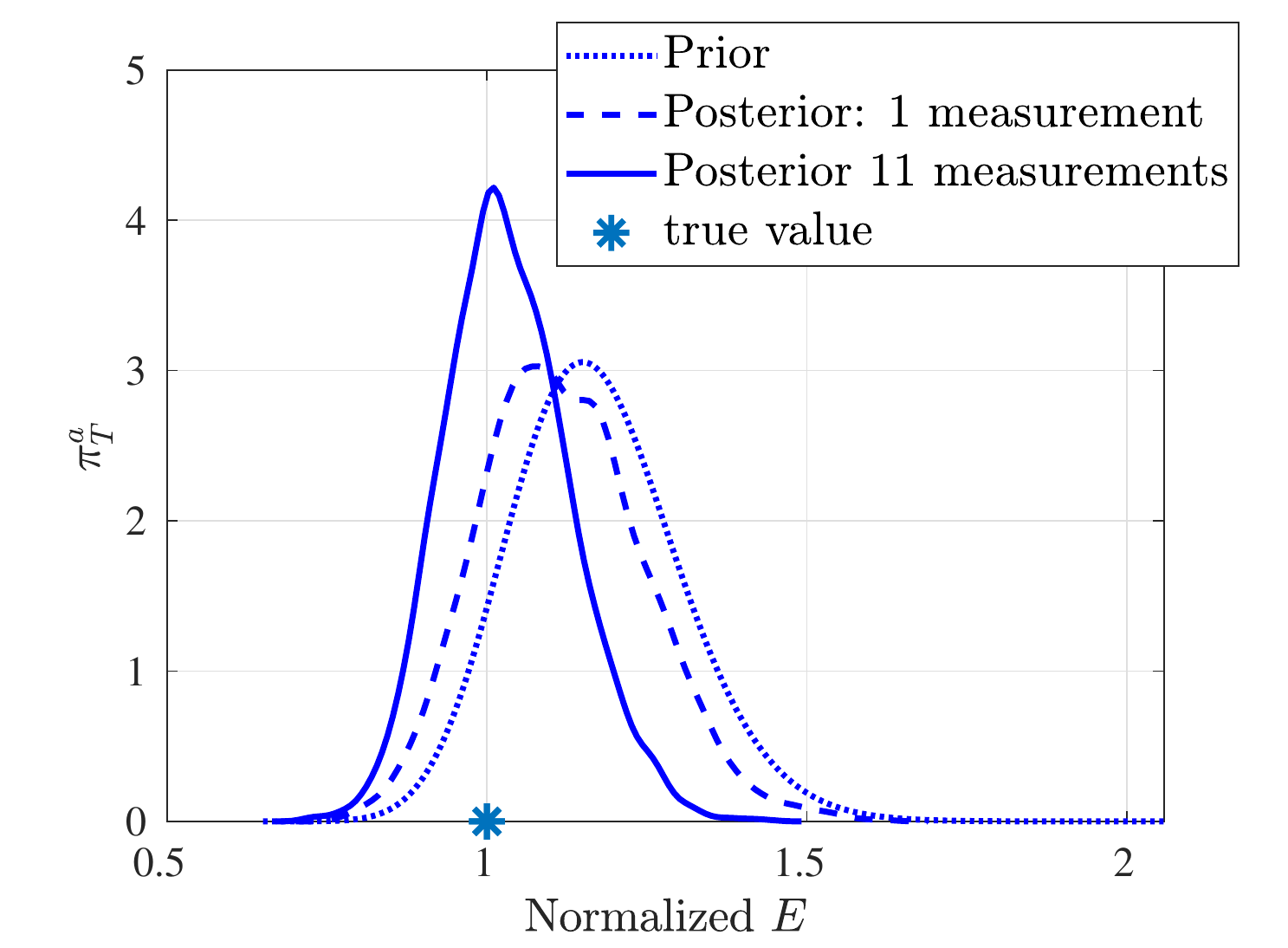}}

     \subfloat[Marginal capacity transfrom pdf of $q$ ]{\includegraphics[scale=0.45]{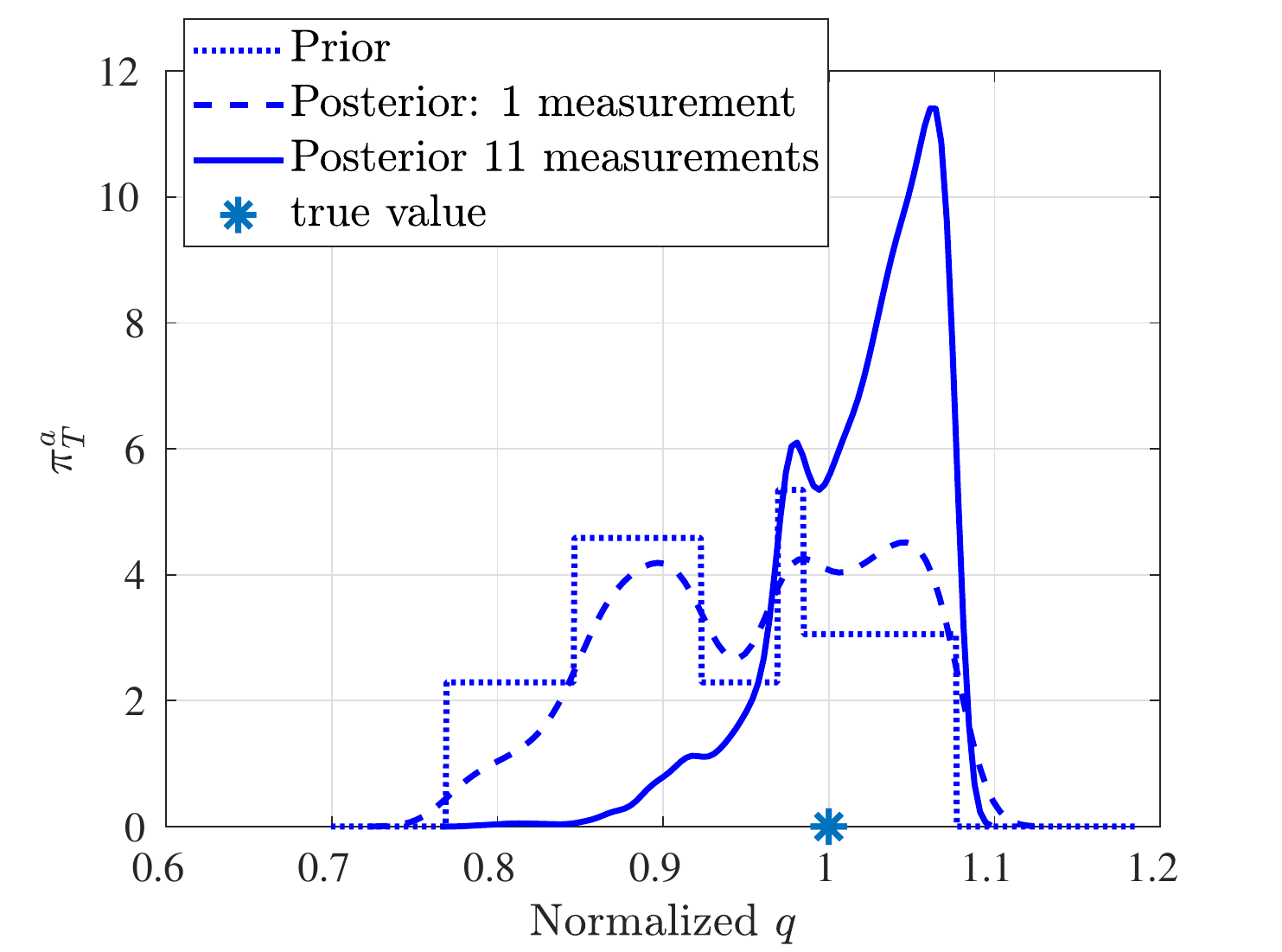}}
     \caption{Updated capacity transform pdf using Bayesian inference. }
     \label{Fig:capacitypdf}
\end{figure}
\subsubsection*{Samples of the discrete random set $\hat{X}^{a,\kappa}$}
Following the approximation method developed in Section.~\ref{Sec:approximation_of_posterior_randomset}, instead of directly sampling the posterior {rs} ${X}^a$ defined in Definition \ref{Def:Dempster's_rule}, we sample its approximation, i.e. the random discrete set $\hat{X}^{a, \kappa}$. 
Using sampling values of $x^{(i)}$ and $u^{(i)} = h(x^{(i)})$ where $i = 1, \dots,\kappa$ from MCMC simulation, the $N$ set-valued samples ${X}^{a,\kappa \; (1)}, \cdots, {X}^{a,\kappa \;(N)}$ of the random discrete set $\hat{X}^{a, \kappa}$ are obtained following the Algorithm~\ref{algorith:approximate_random_set}. No further evaluation of the FE model is required at this step. 

From the samples  ${X}^{a,\kappa\; (1)}, \cdots, {X}^{a,\kappa \;(N)}$, the upper and lower cdf bounds, i.e. $P_{{X}^a} ((-\infty, x))$ and $T_{{X}^a} ((-\infty, x))$ respectively, are evaluated using Eq.~(\ref{Eq:estimate_capacity}). The obtained marginal cdf bounds of the posterior {rs} of $E$ are shown in Fig.~(\ref{Fig:cdf_E}). As it is observed in Fig.~\ref{Fig:cdf_E}, the bounds of cdfs are shrinked after updating. This is explained by the fact that ${X}^a(\omega)$ is a subset of ${X}(\omega)$, see Eq.~(\ref{Eq:Dempster's_rule}). The more data become available, the thinner these bounds. Futhermore, the updated random sets get closer to the truth value similarly with the Bayesian inference of {rv}. 
\begin{figure} [th]
     \centering

     \subfloat[Marginal cdf of $E$ (one measurement) ]{
     \includegraphics[scale=0.43]{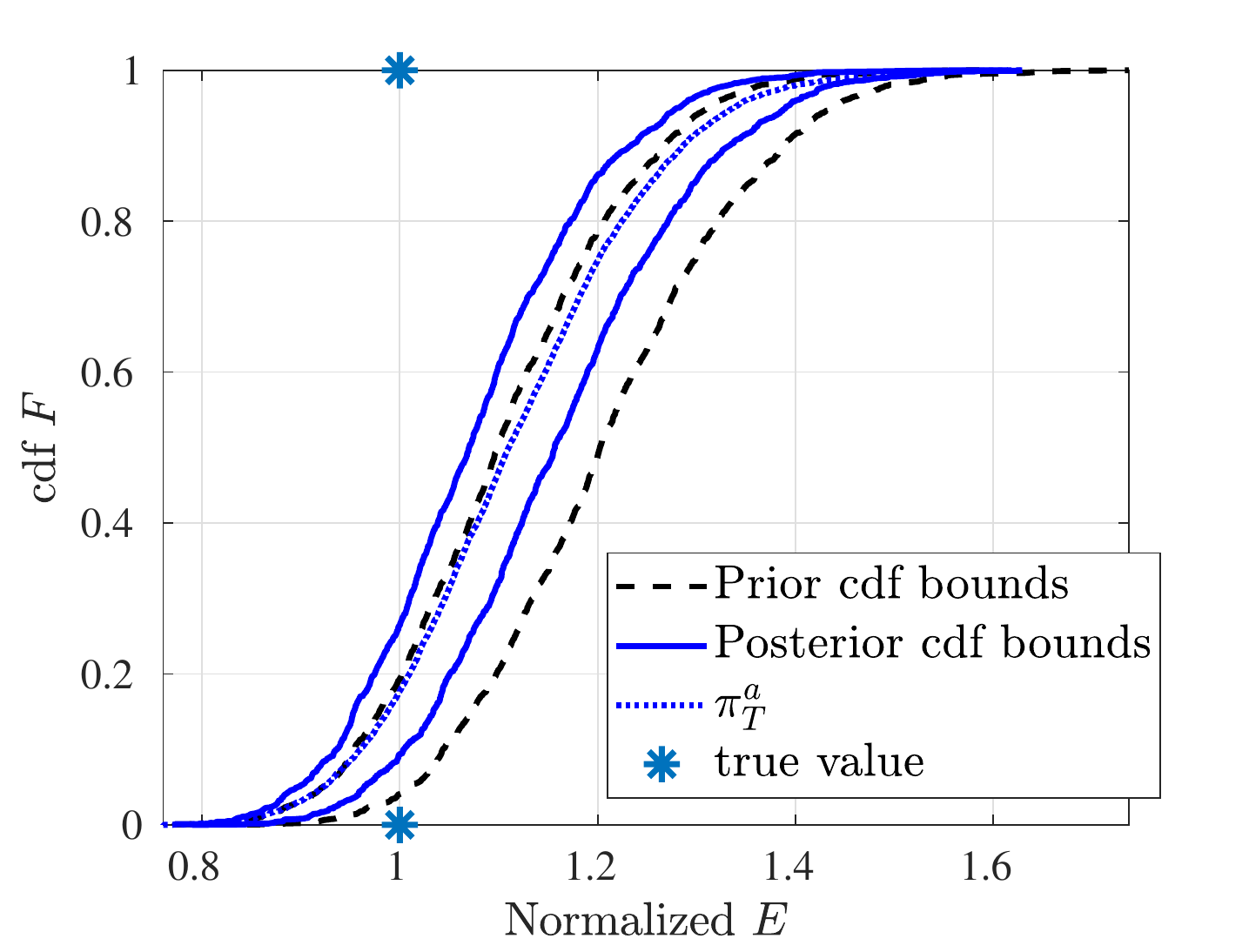}}
     \subfloat[Marginal cdf of $E$ (11 measurements) ]{
     \includegraphics[scale=0.43]{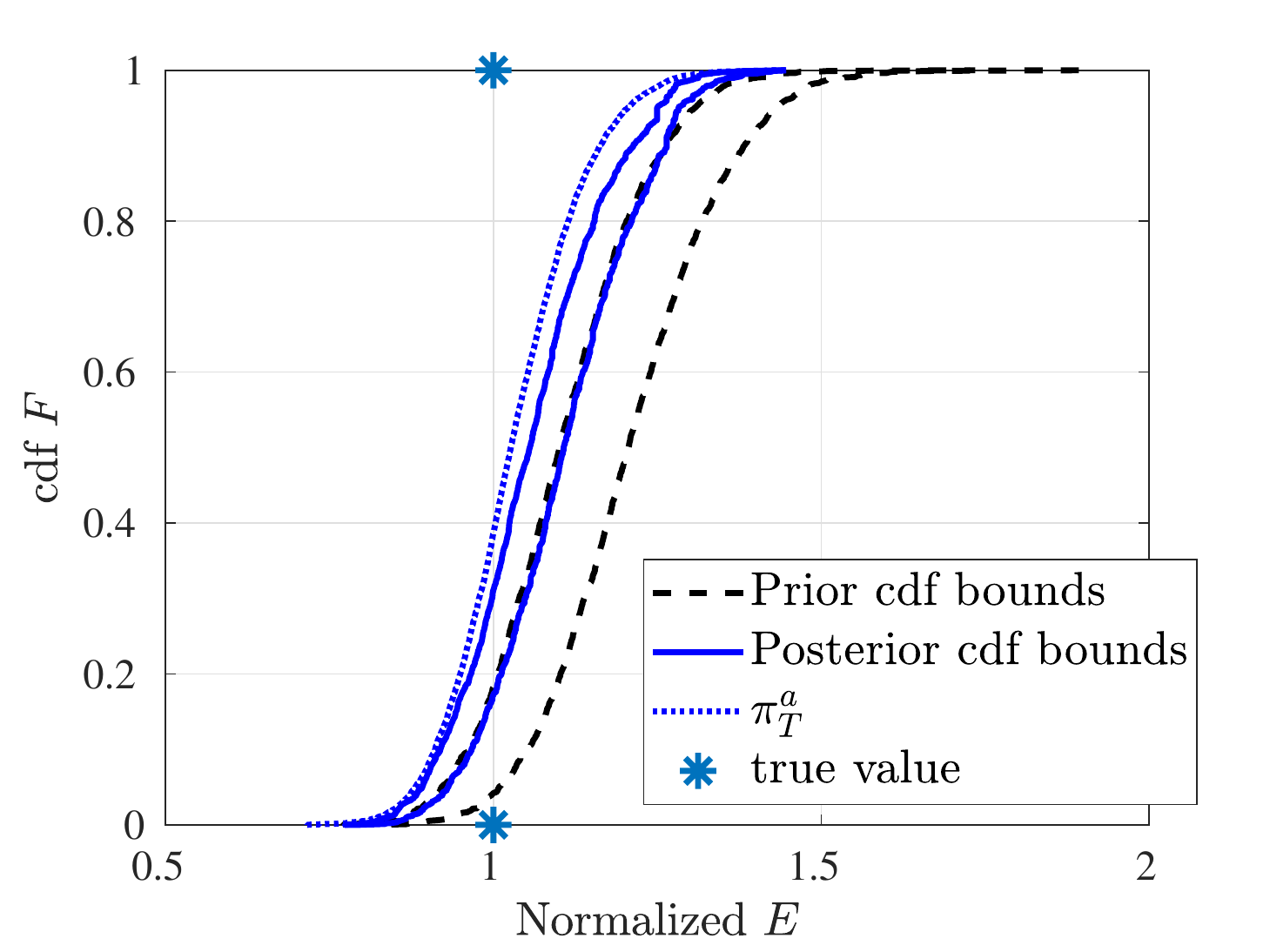}}
     
      \subfloat[Marginal cdf of $q$ (one measurement) ]{
      \includegraphics[scale=0.43]{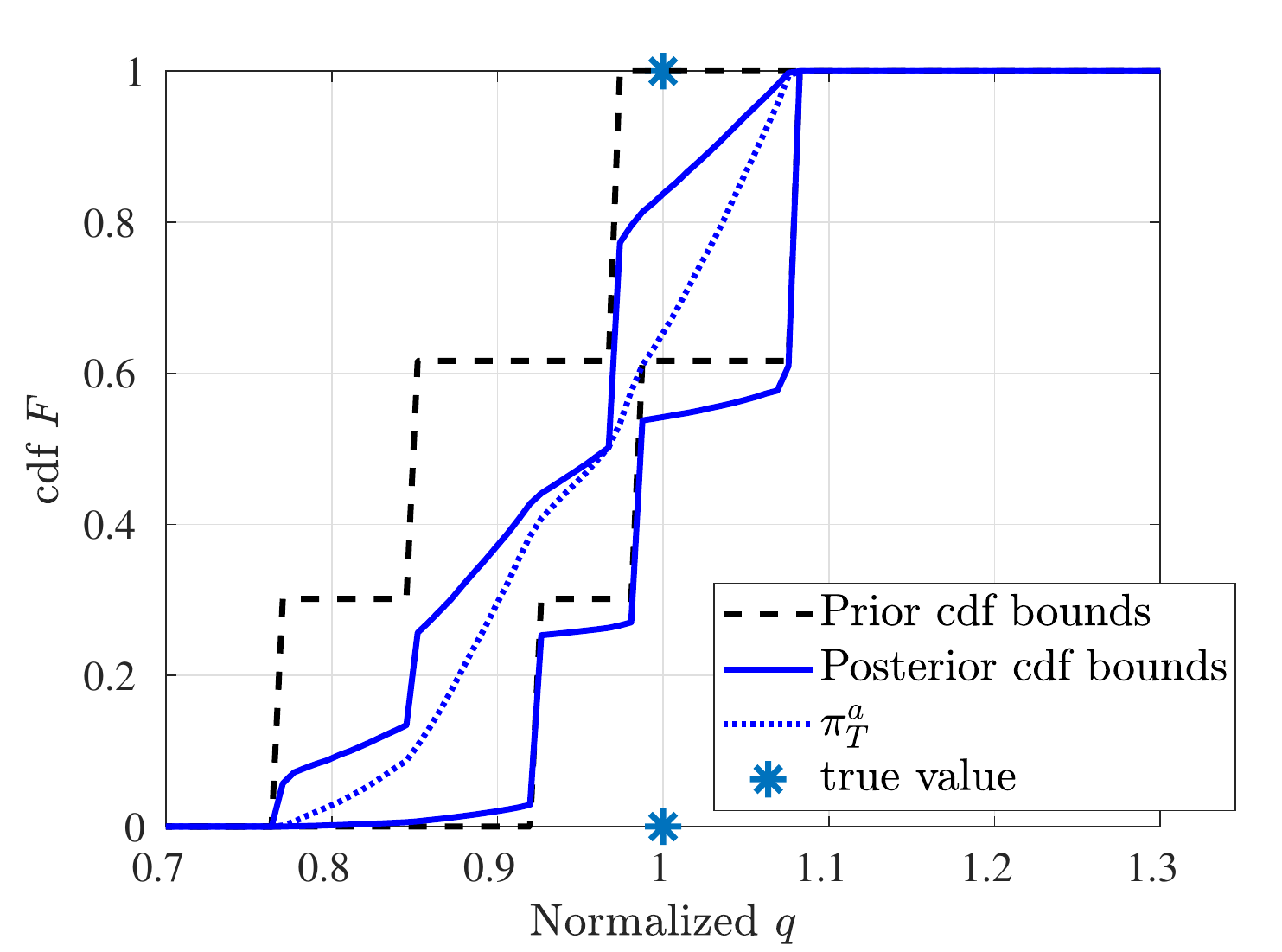}}
      \subfloat[Marginal cdf of $q$ (11 measurements) ]{
      \includegraphics[scale=0.43]{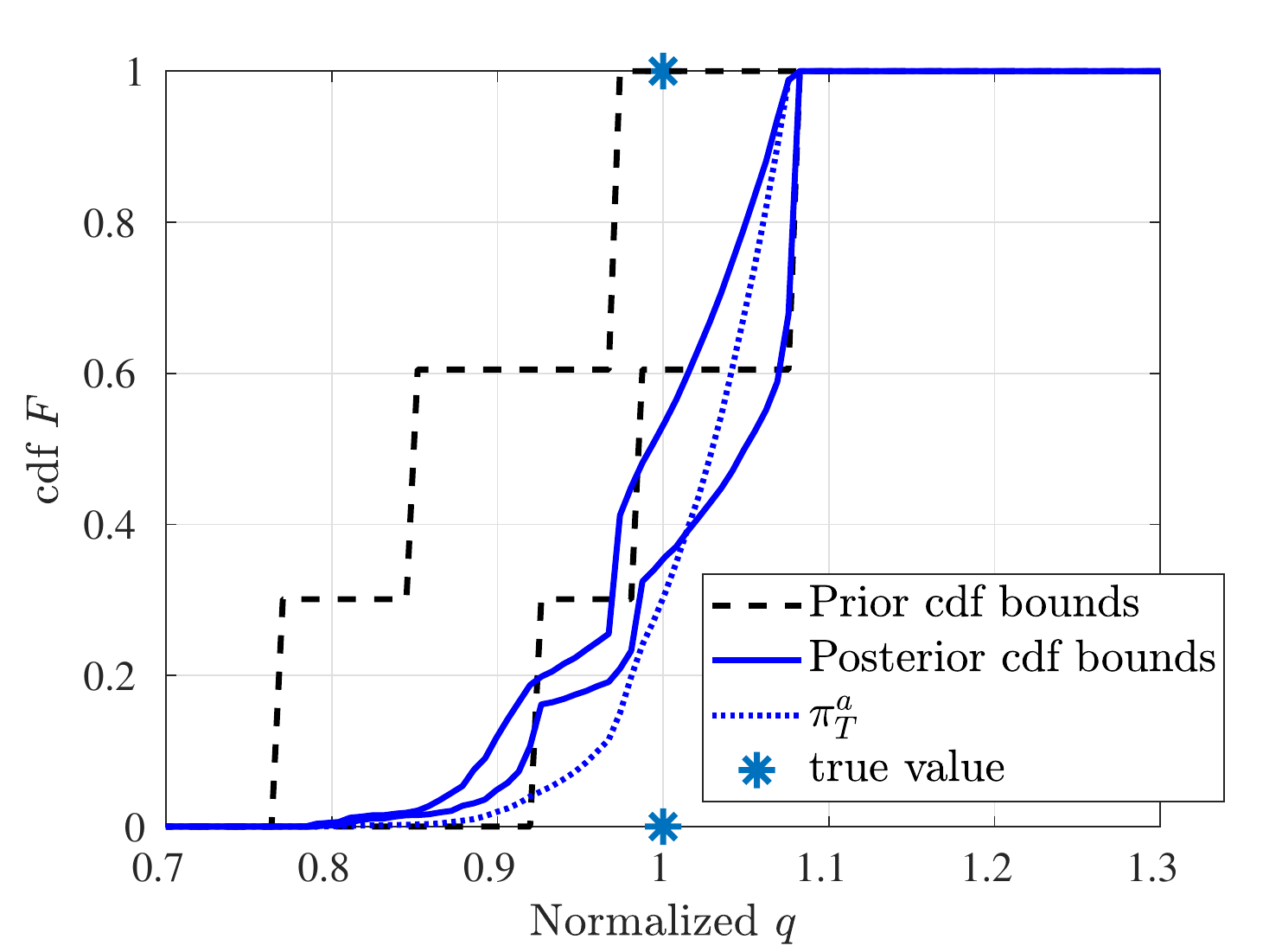}}
     
     \caption{Bounds of the cumulative distribution function (cdf) of the updated {rs}. First line: marginal cdf of $E$, second line: marginal cdf of $q$. First row: only one measurement datum (at a$_1$) is used, second row: all eleven measurement datum are used.}
     \label{Fig:cdf_E}
\end{figure}
 
 \begin{figure} [th]
 	\centering
 	\includegraphics[scale=0.550]{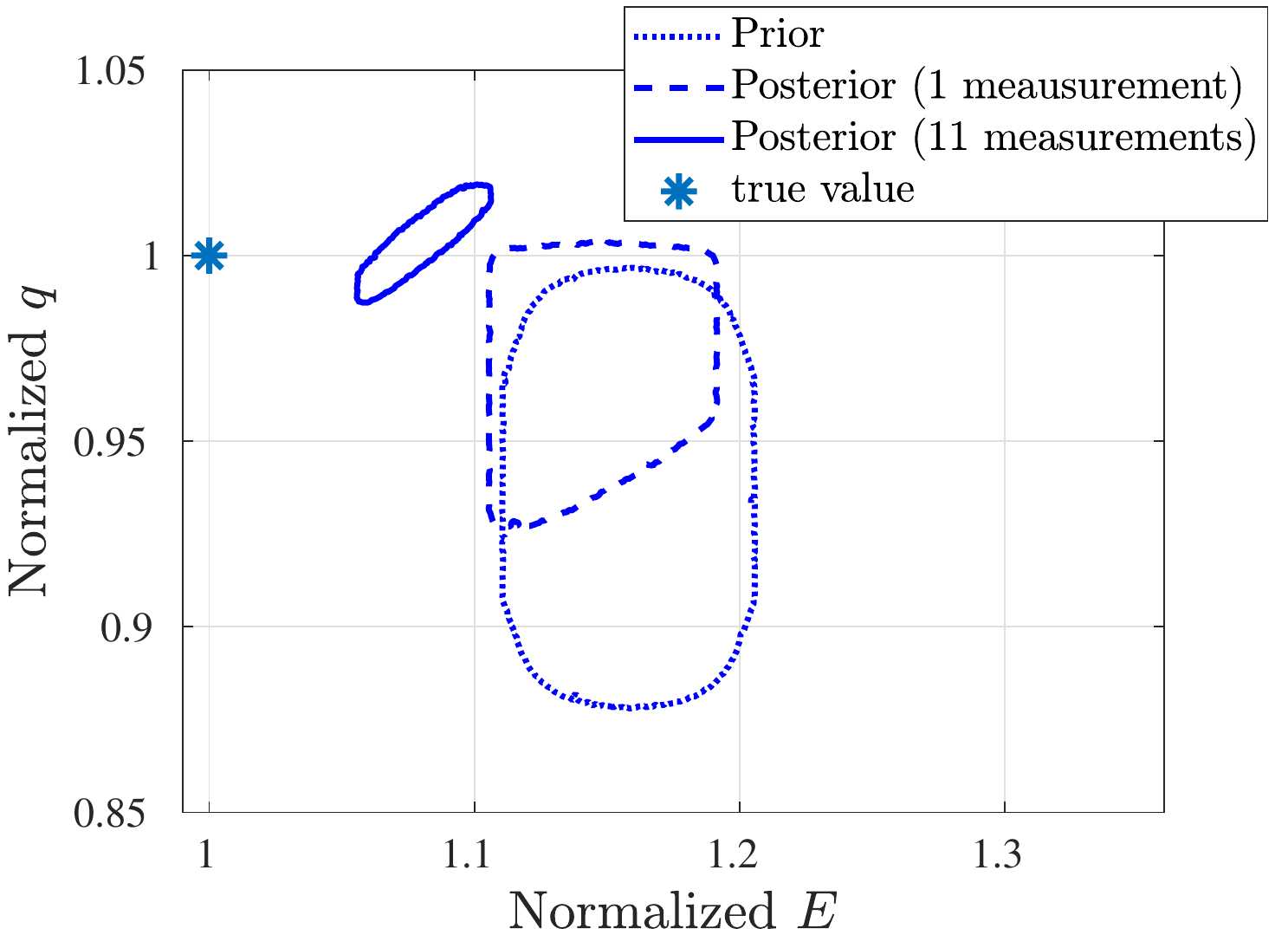}
 	\caption{The boundaries of the selection expectations of prior and posterior random sets obtained in two cases: using one measurement datum (at a$_1$), and using all eleven measurement data.}
 	\label{Fig:convex_hull_expectation}
 \end{figure}

Based on the samples ${X}^{a,\kappa\; (1)}, \cdots, {X}^{a,\kappa \;(N)}$, we compute the boundary of the selection expectation set $\mathbb{E}_S({X}^a)$ defined in Eq.~(\ref{Eq:slection_expectation}). This task requires the evaluation of the support functions  $\gamma(\mathbb{E}_S({X}^a),\nu)$. In this example, we use a non-atomic elementary probability space. Indeed, while a non-atomic elementary probability space can model both the prior random sets of $E$ and $q$, an atomic one can only model the {rs} of $q$. Following the Theorem~\ref{Prop:convex_selection_expectation}, the selection expectation set $\mathbb{E}_S({X}^a)$ is convex, and the support functions $\gamma(\mathbb{E}_S({X}^a),\nu)$ can be computed via the probabilistic expectation of {rv} $\gamma({X}^a (\omega),\nu)$ following Eq.~(\ref{Eq:support_function_slection_expectation}).  From the samples $\hat{X}^{a,\kappa \; (1)}, \cdots, \hat{X}^{a, \kappa \;(N)}$, the approximation of the expectation $\mathbb{E}(\gamma({X}^a (\omega),\nu))$ is evaluated using Eq.~(\ref{Eq:approximate_expectation_set}). The selection expectation boundaries belonging to the prior and the posterior random sets are illustrated in Fig.~\ref{Fig:convex_hull_expectation}. It is observed that, when more measurement data become available, the selection expectation moves toward the true value. It also shrinks in an agreement with the shrinking of the cdf bounds in Fig.~\ref{Fig:cdf_E}. 

\subsubsection*{Convergence analysis} To investigate the convergence of the approximation expressed in Eq.~(\ref{Eq:intersection_MC_approximate}), one can check the mean square error when approximating the expectation of the support function $\mathbb{E} (\gamma({X}^a, \nu))$ by $\mathbb{E} (\gamma(\hat{X}^{a, \kappa}), \nu)$, see Eq.~(\ref{Eq:approximate_expectation_set}). That mean square error is defined as
\begin{equation}
MSE_\nu (\kappa)= \Bigl( \int_{\mathbb{X}^\kappa}
\bigr[\mathbb{E} (\gamma({X}^a, \nu)) - \mathbb{E} (\gamma(\hat{X}^{a, \kappa}, \nu ) ) \bigr]^2
\prod_{i=1}^{\kappa}\pi_T^a(x^{(i)}) dx^{(i)}\Bigl)^{1/2}
\end{equation}
Because $\mathbb{E} (\gamma({X}^a, \nu))$ is unknown $MSE_\nu$ is approximated as
\begin{equation}
MSE_\nu (\kappa) \approx \Bigl( \int_{\mathbb{X}^\kappa}
\bigr[ \mathbb{E} (\gamma(\hat{X}^{a, \kappa_\infty}, \nu)) - \mathbb{E} (\gamma(\hat{X}^{a, \kappa}, \nu ) ) \bigr]^2
\prod_{i=1}^{\kappa}\pi_T^a(x^{(i)}) dx^{(i)} \Bigl)^{1/2},
\end{equation}
where $\kappa_\infty \gg \kappa$. The integration of the mean square error $MSE_\nu$ can be computed using MC method. 

The normalized mean square errors $MSE_\nu$ of different values $\nu$ are illustrated in Fig.~\ref{Fig:convergen_analysis}. It can be observed that with $\kappa = 2000$ samples $x^{(i)}$ of the pdf $\pi_T^a$, the normalized mean square error $MSE_\nu$ is smaller than 3\%. 
 \begin{figure} [th!]
 	\centering
 	\includegraphics[scale=0.55]{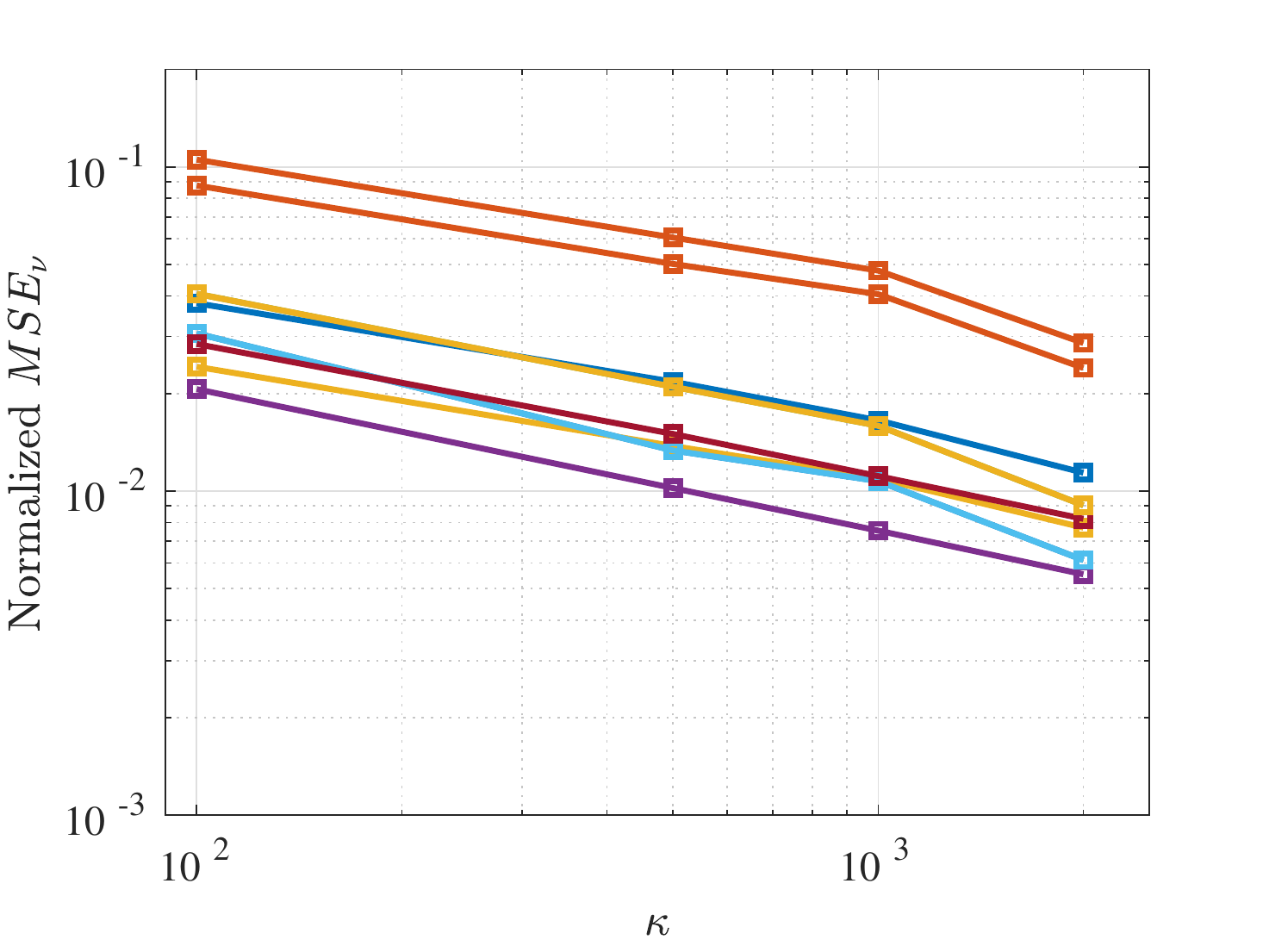}
 	\caption{The normalized mean square error $MSE_\nu$ for different values of $\nu$ (the illustrated case is when all data are used for the inference).}
 	\label{Fig:convergen_analysis}
 \end{figure}
\section{Conclusion}
\label{Sec:conclusion}
In this work, we develop a framework for the non-deterministic inference using random set models. 
The inference rule is based on Dempster's rule of combination. We show that the proposed reference is a generalisation of the Bayesian one. The posterior {rs} is approximated using random discrete set whose the domain is the set of samples of a proposed distribution. The capacity transform pdf of the posterior {rs} is chosen as the proposed distribution. With this choice, the required samples of the proposed distribution can be obtained using the methods developed in Bayesian framework, e.g. MCMC. The computational burden is hence comparable with those methods in Bayesian inference. 

A {rs} can equivalently formulate other uncertainty modelling methods, e.g. {rv}, set of possible values, mass belief function in the evidence theory, and probability box. Therefore, the developed inference framework can be applied for all these cases as well as their combinations. We have demonstrated this advantage in a numerical example in Section~\ref{Sec:numerical_example}.

Since the computation method for the set-valued selection expectation has been developed, the framework can be extended toward decision making theory. 
In addition, the special property of the capacity transform pdf is promising when dealing with data assimilation involving random sets.

\appendix

\section{Dempster's combination rule}\label{appendix:Dempster_rule}
Let $M$ and $M^d$ be two mass belief functions, see Section~\ref{Sec:DS_theory} for their definition. The Dempster's rule to combine the two belief mass functions $M$ and $M^d$ is given as  
\begin{equation}
M^{a} (\mathcal{X}) = \dfrac{1}{1-K} \quad \sum_{\mathcal{X}_i \in \mathfrak{X}_M,\; \mathcal{X}_j \in \mathfrak{X}_{M^d}} M(\mathcal{X}_{i}) M^d(\mathcal{X}_{j})\; \mathbf{1}({\mathcal{X}_{i}\cap \mathcal{X}_{j} = \mathcal{X}})
\end{equation}  
where $\mathcal{X} \neq \emptyset$, and $K$ is a measure of the amount of the conflict between $M$ and $M^d$
$$ K =  \sum_{\mathcal{X}_i \in \mathfrak{X}_M,\; \mathcal{X}_j \in \mathfrak{X}_{M^d}} M(\mathcal{X}_{i}) M^d(\mathcal{X}_{j})\mathbf{1}(\mathcal{X}_i \cap \mathcal{X}_j = \emptyset)$$

In Dempster's combination rule, a non empty set $\mathcal{X}$ has a positive belief mass after updated, i.e $M^a(\mathcal{X})> 0$, only if there exist a pair $\mathcal{X}_{i}\in \mathfrak{X}_M$, $\mathcal{X}_{j}\in \mathfrak{X}_{M^d}$ such that 
 $\mathcal{X}_{i}\cap \mathcal{X}_{j} = \mathcal{X} $. 
The inference methods described in Definition~\ref{Def:Dempster's_rule} is equivalent to Dempster's combination rule when the random sets ${X}$ and ${X}^d$ are resulted from the mass belief functions. 
\section{Comparison of two inference strategies of {rs} on a simple example}
\label{appendix:roburst_bayesian_vs_Dempster}
In this section we compare the two inference strategies: (i) using Dempster's rule as discussed in Section~\ref{Sec:inversion_random_set} and (ii) using Bayesian interference on the set of selection random variables. 
Let $\varOmega = \{\omega_1, \omega_2\},\;  \mathbb{P}(\omega_1) + \mathbb{P} (\omega_2)=1 $, $\mathfrak{X} = \{\mathcal{X}_1, \mathcal{X}_2\}$. 
The {rs} represents the prior knowledge about the variable $x$ is given by
 \begin{equation} 
 {X}(\omega_1) = \mathcal{X}_1, \quad  {X}(\omega_2) = \mathcal{X}_2.
 \end{equation}
Assuming that we have a direct measurement about $x$ which give us the following information
$x \in \{x_1, x_2\}$ where $x_1 \in  \mathcal{X}_1 \backslash \mathcal{X}_2 $ and $x_2 \in  \mathcal{X}_2 \backslash\mathcal{X}_1 $. 
The {rs} induced by this measurement is 
$$ {X}^d (\omega)= \{x_1, x_2\} \quad \forall \omega \in \varOmega.$$ 
\subsection{Inference using Dempster's rule}
Using Dempster's rule, the updated posterior {rs} is obtained as
\begin{equation}
{X}^{a}(\omega_1) = x_1, \;{X}^a(\omega_2) = x_2.
\end{equation}
In other words, the sets $\mathcal{X}_1$ and $\mathcal{X}_2$ shrink to become $x_1$ and $x_2$ respectively. 
The {rs} ${X}^{a}$ is a {rv}, and its distribution function is given as
\begin{equation}
P_{{X}^{a}}(x_1) = \mathbb{P}(\omega_1), \;P_{{X}^{a}}(x_2) = \mathbb{P}(\omega_2).
\label{Eq:Dempster_rule_examples}
\end{equation}
\subsection{Inference using Bayes' rule on the set of selection random variables}
Let $P_x$ be a probability distribution of a selection {rv} $x(\omega)$. 
Using Bayes's rule, the updated distribution $P_{x^a}$ of the prior distribution $P_x$  is given as
 \begin{equation}
 P_{x^{a}}(x_1) = \dfrac{P_x(x_1) } {P_x(x_1) + P_x(x_2)}, \;P_{x^a}(x_2) = \dfrac{P_x(x_2) } {P_x(x_1) + P_x(x_2)} \quad \text{if} \; P_x(x_1) + P_x(x_2) > 0.
 \end{equation}
In cases that $P_x(x_1) + P_x(x_2) = 0$, no update is possible. These are two special cases, 
 \begin{equation}
 P_{x^a}(x_1) = 1,\; P_{x^a}(x_2) = 0, \quad \text{if} \; P_x(x_1) >0,\; P_x(x_2) = 0,
 \end{equation}
and
\begin{equation}
P_{x^a}(x_1) = 0, \; P_{x^a}(x_2) = 1, \quad \text{if} \; P_x(x_1) =0, \; P_x(x_2) > 0.
\label{Eq:Bayes_rule_examples}
\end{equation}
Based on that update, the posterior knowledge is represented as 
 \begin{equation}
 0 \leq   P_{x^a}(x_1), P_{x^a} (x_2) \leq 1,\quad \text{and} \; P_{x^a}(x \notin \{x_1, x_2\}) =0.
 \end{equation}

This conclusion is independent on the prior probability $P_{{X}}(\mathcal{X}_1), P_{{X}}(\mathcal{X}_2)$ and is therefore less informative compared to the one given by Dempster's combination rule described in Eq.~(\ref{Eq:Dempster_rule_examples}). That problem is due to the fact that the selection random variables are updated independently without any interaction to each other.
While in the inference using Dempster's rule, a {rs} is considered as a single information. 
 
\section{Proof of Theorem \ref{lemma:bayes_dempster_rules}}\label{appendix:Proof_bayes_dempster_rules}
The pdf $\pi{^a}$ of {rv} $x^a(\omega)$ is given by
 \begin{equation}
  \pi{^a}(x) := \int_{\varOmega} \delta (x - x^{a}(\omega))  \mathbb{P}^a(\dint \omega)
 \end{equation}
 where $\delta$ is Dirac delta function.
Insert $\mathbb{P}^a$ in Eq.~(\ref{Eq:DS_rule_to_Bayes}) to the expression of $\pi^a$ yields
 \begin{equation}
  \pi{^a} (x) = \dfrac{1}{C} \int_{\varOmega} \delta (x - x(\omega))  \; \mathbf{1}_{{X}^d(\omega)}(x)  \mathbb{P}(\dint \omega).
 \end{equation}
 where $C$ is a constant given by
 $$ C = \int_{\varOmega} \mathbf{1}_{{X}^d(\omega)}(x(\omega))  \mathbb{P}(\dint \omega) $$
As $x(\omega)$ and ${X}^d(\omega)$ are independent we have
\begin{equation}
\pi{^a} (x) = \dfrac{1}{C}  \int_{\varOmega} \delta (x- x(\omega)) \mathbb{P}(\dint\omega) \int_{\varOmega}  \mathbf{1}_{{X}^d(\omega)}(x) \mathbb{P}(\dint \omega)
= \dfrac{1}{C} \pi(x) \mathcal{L}(x),
\label{Eq:Bayes_rule_from_Dempster_rule_appendix}
\end{equation}
where $\pi(x)$ is the pdf of r.v. $x(\omega)$, and 
 $$ C = \int_\mathbb{X} \int_{\varOmega} \mathbf{1}_{{X}^d(\omega)}(x)   \mathbb{P}(\dint \omega) \pi (x)  \dint x = \int_\mathbb{X}  \mathcal{L}(x) \pi(x) \dint x.$$ 
The equation (\ref{Eq:Bayes_rule_from_Dempster_rule_appendix}) expresses the Bayesian inference. 

\section{Proof of Theorem \ref{TH:condtion_of_proposed_pdf}} \label{appendix:condtion_of_proposed_pdf}
The Hausdorff distance betwen $\mathcal {X}$ and $\hat{\mathcal {X}}$  is defined as
\begin{equation}
d_{{{\mathrm  H}}}(\mathcal {X},\hat{\mathcal {X}}) = 
\max\{
\,\sup_{ x\in \mathcal {X} }\inf_{ x'\in \hat{\mathcal{X}} }d(x,x'),
\,\sup_{ x'\in \hat{\mathcal {X}} }\inf_{ x\in \mathcal{X} }d(x,x')\, 
\},
\end{equation}
where $d(x, x')$ is the Euclid distance between two points $x$ and $x'$. 
Since $\hat{\mathcal {X}} \subset  {\mathcal {X}}$, $d_{{{\mathrm  H}}}(\mathcal {X},\hat{\mathcal {X}})$, is simplified as
\begin{equation}
d_{{{\mathrm  H}}}(\mathcal {X},\hat{\mathcal {X}})=\sup _{{x\in \mathcal {X}}}\inf _{{x'\in \hat{\mathcal {X}}}}d(x,x').
\end{equation}
Let $B(x,r)$ be a ball centred at $x \in \mathcal {X}$ and radius $r>0$.
The probability of the event $\hat{\mathcal {X}} \cap  \; (B(x,r) \cap \mathcal {X}) =\emptyset$, i.e. $x^{(i)} \notin B(x,r) \cap \mathcal {X}$ $\forall \; i \in \{1, \dots,\kappa\} $,  is 
\begin{equation}
P_{\pi^e} (\hat{\mathcal {X}} \cap  \; (B(x,r) \cap \mathcal {X}) =\emptyset) = \bigr(1- \int_{B(x,r) \cap \mathcal {X} }\pi^{e}(x)dx\bigr) ^\kappa.
\end{equation}
Since
$\pi^{e}(x)dx > 0$ $\forall x\in \mathcal {X}$ and $\mathcal {X} =\text{cl}(\text{int}(\mathcal{X}))$, we have $\int_{B(x,r) \cap \hat{\mathcal{X}}}\pi^{e}(x)dx>0$ and 
$$
\lim_{\kappa \rightarrow \infty}\bigr(1- \int_{B(x,r)\cap \mathcal {X}}\pi^{e}(x)dx\bigr) ^ \kappa \rightarrow 0, \; \forall x\in \mathcal {X}.
$$
For all $r > 0$, there exists $n_b< \infty$ points $x_{(i)}$ in $\mathcal{X}$ such that $\mathcal{X} \subset \cup_{i=1}^{n_b} B(x_{(i)}, r)$. The probability that the Hausdorff distance between the set $\mathcal {X}$ and the set $\hat{\mathcal {X}}$ is larger than $2r$ converges to zero as $\kappa \rightarrow \infty$, i.e. 
\begin{equation}
P_{\pi^e} (d_H( \mathcal {X}, \hat{\mathcal {X}}) \geq 2r) \leq \sum_{i = 1}^{n_b} \bigr(1- \int_{B(x_{(i)},r) \cap \mathcal {X}}\pi^{e}(x)dx\bigr) ^ \kappa \; dx  \rightarrow 0, \; \forall r>0.
\end{equation}
In other words, the probability $P_{\pi^e}$ of the event $\lim_{\kappa \rightarrow \infty}(d_H(  \mathcal{X}, \hat{\mathcal{X}})) = 0$ is one.

\section*{Acknowledgement}
Funded by the Deutsche Forschungsgemeinschaft (DFG) through the SPP-1886.

\bibliographystyle{ieeetr}  
\bibliography{library}
\end{document}